\documentclass[a4paper,10pt]{amsart}
\usepackage{amsmath}
\usepackage{amssymb}
\usepackage{amsxtra}
\usepackage{amsthm,amsfonts,latexsym}
\usepackage{epic}
\usepackage{eepic}
\usepackage{fancybox}
\usepackage{boxedminipage}
\usepackage{calc}
\usepackage{tabularx}
\usepackage{graphicx,color}
\usepackage{longtable,booktabs}
\usepackage{supertabular,longtable}
\usepackage{paralist}
\usepackage{amsmath}
\usepackage{amsfonts}
\usepackage{amssymb}
\usepackage{amsthm}
\usepackage[english]{babel}
\usepackage[latin1,ansinew]{inputenc}
\usepackage{a4wide}
\usepackage{pdfpages}
\usepackage{color}
\newtheorem{thm}{Theorem}

\newtheorem{remark}{Remark}

\newcommand{\bea}{\begin{eqnarray}}
\newcommand{\eea}{\end{eqnarray}}
\newcommand{\beann}{\begin{eqnarray*}}
\newcommand{\eeann}{\end{eqnarray*}}
\newcommand{\ba}{\begin{array}}
\newcommand{\ea}{\end{array}}

\newcommand{\beq}{\begin{equation}}
\newcommand{\eeq}{\end{equation}}
\newcommand{\be}{\begin{equation}}
\newcommand{\ee}{\end{equation}}

\accentedsymbol{\hcirc}{ {\overset{\scriptscriptstyle \circ }{ {\rm H} }}}

\newcommand{\R}{\mathbb{R}}

\newcommand{\del}[2]{ \partial^{#1}_{#2} }
\newcommand{\divv}{ \nabla \!\! \cdot \!}
\newcommand{\bu}{{\bf u}}

\newcommand{\ska}[3]{ \left ( #1 \, | \, #2 \right )_{ #3 } }
\title{Reductions of the Navier-Stokes-Allen-Cahn \\ and \\ 
the Navier-Stokes-Cahn-Hilliard equations}
\author{Heinrich Freist\"uhler and Matthias Kotschote}
\date{October 26, 2013}
\begin{document}
\begin{abstract}
This paper studies two well-known models for two-phase fluid flow at constant temperature,   
the isothermal Navier-Stokes-Allen-Cahn and the isothermal Navier-Stokes-Cahn-Hilliard equations,
both of which consist of equations for the (total) fluid density $\rho$,
the (mass-averaged) velocity $\bu$ and the concentration (of one of the phases,) $c$.
Assuming in either case that both phases are incompressible with different 
densities, each of the models is shown to reduce to a system of evolution equations 
in $\rho$ and $\bu$ alone. In the case of the Navier-Stokes-Allen-Cahn model,
this reduced system is the classical Navier-Stokes-Korteweg model. The reduced
system resulting from the Navier-Stokes-Cahn-Hilliard equations is a novel 
`integro'-differential system in which a non-local operator acts on the divergence
of the velocity.   
\end{abstract}
\maketitle
\parindent=0cm
\section{With phase transformation} 
\label{sec:nsac}
We consider the Navier-Stokes-Allen-Cahn system (NSAC) of Blesgen \cite{B},
\begin{alignat}{3}
\partial_t \rho + \divv ( \rho \bu ) & = 0, & \quad 
\notag
\\
\partial_t (\rho \bu)  + \divv ( \rho \bu \otimes\bu) & = 
\divv ( \mathcal{S}(\bu)  + \mathcal{P}(c)) 
\label{eq:nsac}
\\
\partial_t (\rho c)  + \divv (\rho c\, \bu ) & =
\delta^{-1/2}(\rho q
+
\divv \left(\delta \rho \nabla c \right))
\notag
\end{alignat}
with the Cauchy stress tensor $\mathcal{S}$ and non-hydrodynamic tensor $\mathcal{P}$
\begin{equation} \label{eq:S-P}
\begin{split}
\mathcal{S}(\bu) & = 2 \mu \mathcal{D}(\bu) + \lambda \divv \bu \mathcal{I}, 
\quad \mathcal{D}(\bu) = \frac{1}{2} \big( \nabla\bu + (\nabla \bu)^T \big) 
\\
\mathcal{P}(c) & = - p \mathcal{I} - \theta \delta \rho \nabla c \otimes \nabla c
\end{split}
\end{equation}
and fixed temperature $\theta>0$.
The density $\rho$ and the production rate $q$ are given by 
\begin{equation} \label{eq:rho}
\rho^{-1}=\tau=G_p(p,c) \quad\hbox{and}\quad q= - \frac{1}{\theta} G_c(p,c) . 
\end{equation}
Here $G$ denotes the \textit{extended} Gibbs energy assumed to be of the form
\begin{equation}\label{eq:Gibbs}
G(p,c,\nabla c) = \hat G(p,c) + \theta\big( W(c) + \frac{\delta}{2} |\nabla c|^2 \big)
\end{equation}
with
\begin{align} \label{eq:LT}
\hat{G}(p,c) & =(c\tau_1+(1-c)\tau_2)p
\end{align}
and $\tau_1,\tau_2 > 0$ are constants with
\begin{equation}
\tau_1\neq\tau_2. 
\end{equation}
Note that \eqref{eq:Gibbs}, \eqref{eq:LT} is the Gibbs energy used by Lowengrub and Truskinovsky 
in their pioneering work on two-fluid mixtures \cite{LT}, in which they formulate and treat 
the corresponding Navier-Stokes-Cahn-Hillard system (that we consider below). \\
\begin{thm} \label{thm:nsac}
NSAC can 
be written as the Navier-Stokes-Korteweg system
\begin{equation} \label{eq:NSK-1}
\begin{split}
\partial_t \rho + \divv ( \rho \bu ) & = 0, 
\\
\partial_t (\rho \bu)  + \divv ( \rho \bu \otimes\bu )
- 
\divv (\mathcal{S}_\delta(\bu) + \mathcal{K}(\rho)) & = 0 
\end{split}
\end{equation}
with
\begin{align}
\mathcal{S}_\delta (\bu) & = 2 \mu \mathcal{D}(\bu) + \lambda_*(\rho) \divv \bu \, \mathcal{I}, 
\quad \lambda_*(\rho):= \lambda + \frac{\delta_*}{ \delta^{1/2} \rho }, \quad \delta_* := \frac{\theta \delta}{ (\Delta \tau)^{2}}, 
\quad \Delta \tau := \tau_1 - \tau_2, \label{eq:S}
\\
\mathcal{K}(\rho) & = - \rho^2 \psi_\rho \, \mathcal{I} + \rho \divv ( \kappa(\rho) \nabla \rho ) \, \mathcal{I} 
- \kappa(\rho) \nabla \rho \otimes \nabla \rho,\label{eq:K}
\\
\psi(\rho) & = R(\rho) + \frac{\kappa(\rho)}{2 \rho} |\nabla \rho|^2, \quad R(\rho) := \theta W(\hat{c}(\rho)), \quad 
\kappa(\rho) := \frac{\delta_*}{ \rho^{3} }, \quad \hat{c}(\rho) := \frac{1 - \tau_1 \rho}{ \Delta \tau\ \rho }. \label{eq:psi}
\end{align}
\end{thm}
\begin{proof} 
First, note that by relation \eqref{eq:LT} the unknown $c$ can be regarded as a function depending 
on $\rho$. We therefore set 
\begin{equation} \label{eq:hat-c}
c = \hat{c}(\rho) = \frac{1/\rho - \tau_1}{\tau_1 - \tau_2}, \quad \rho c = \tilde{c}(\rho) = \frac{1 - \tau_1 \rho}{\tau_1- \tau_2} 
\end{equation}
with derivatives 
\begin{equation*}
\tilde{c}'(\rho) = \frac{1}{\Delta \tau}, \quad  \hat{c}'(\rho) =  - \frac{1}{\Delta \tau} \frac{1}{\rho^2}. 
\end{equation*}
Inserting these relations into the Allen-Cahn equation we obtain
\begin{equation*}
\del{}{t} (\rho c) + \divv ( \rho c u) \equiv  \tilde{c}'(\rho) \big( \del{}{t} + \nabla \rho \cdot u) + \tilde{c} \divv u
= (\tilde{c} - \rho \tilde{c}'(\rho) ) \divv u = \frac{1}{\Delta \tau} \divv u 
\end{equation*}
as well as
\begin{align*}
 \delta^{-1/2}(\rho q + \divv \left(\delta \rho \nabla c \right))
\equiv - \delta^{-1/2} \rho \left( \frac{\Delta \tau}{\theta} p + W'(c) \right)
-
\delta^{-1/2} \divv \left( \frac{\delta}{\Delta \tau \rho} \nabla \rho \right).
\end{align*}
From these relations we are able to find en explicit representation of the pressure,
\begin{equation*}
\begin{split}
p & =  - \frac{ \delta^{1/2} \theta}{ (\Delta \tau)^2 \rho } \divv u - \frac{\theta}{\Delta \tau} W'(\hat{c}(\rho)) 
- \frac{1}{\rho} \divv \left( \frac{\theta\delta}{(\Delta \tau)^2 \rho} \nabla \rho \right) 
\\
& =  - \frac{ \delta_*}{ \delta^{1/2} \rho } \divv u + \rho^2 R'(\rho) 
- \frac{1}{\rho} \divv \left( \frac{\delta_* }{\rho} \nabla \rho \right), \quad R(\rho) := W(\hat{c}(\rho)).
\end{split}
\end{equation*}
This yields
\begin{align*}
\divv ( \mathcal{S}(\bu) + \mathcal{P}(c) ) & = \divv \mathcal{S}_\delta(u) - \nabla \cdot \left( \rho^2 R'(\rho) \mathcal{I} 
- \rho^{-1} \divv \left( \delta_* \rho^{-1} \nabla \rho \right) \mathcal{I}  + \theta \delta \rho \nabla c \otimes \nabla c \right) 
\\
& = \divv \mathcal{S}_\delta(u) + \nabla \cdot \left( - \rho^2 R'(\rho) \mathcal{I} 
+ \rho^{-1} \divv \left( \delta_* \rho^{-1} \nabla \rho \right) \mathcal{I}  - \delta_* \rho^{-3} \nabla \rho \otimes \nabla \rho \right), 
\end{align*}
where we have used \eqref{eq:hat-c} and $\mathcal{S}_\delta(\bu)$ is given by \eqref{eq:S}. 
Let $\kappa(\varrho) = \delta_* \varrho^{-3}$. The second last term can be rewritten as 
\begin{align*}
\nabla \cdot \left( \varrho^{-1} \divv ( \varrho^2 \kappa(\varrho) \nabla \varrho) \mathcal{I} \right) 
& =
\nabla \left( \varrho \divv ( \kappa(\varrho) \nabla \varrho) + \varrho^{-1} \kappa(\varrho) \nabla \varrho \cdot 2 \varrho \nabla \varrho \right)
\\
& = \nabla ( \varrho \divv ( \kappa(\varrho) \nabla \varrho) + 2 \kappa(\varrho) |\nabla \varrho|^2 )
\\
& = \nabla \big( \varrho \divv ( \kappa(\varrho) \nabla \varrho) - \varrho^2 \del{}{\varrho} 
\big( \frac{\kappa(\varrho)}{2 \varrho} |\nabla \varrho|^2 \big) \big).
\end{align*}
Using this identity we find
\begin{multline*}
\divv  \big ( - \varrho^2 R'(\varrho) \mathcal{I} + \varrho^{-1} \divv \left( \delta_* \varrho^{-1} \nabla \varrho \right) \mathcal{I} 
- 
\delta_* \varrho^{-3} \nabla \varrho \otimes \nabla \varrho \big)
\equiv
\\
\divv \left( - \varrho^2 \del{}{\varrho} \psi \mathcal{I} + \varrho \divv ( \kappa(\varrho) \nabla \varrho ) \mathcal{I} 
- \kappa( \varrho) \nabla \varrho \otimes \nabla \varrho \right) =: \divv \mathcal{K}
\end{multline*}
with $\psi$, $R$, and $\kappa$ given as in \eqref{eq:psi}. The tensor $\mathcal{K}$ is exactly the Korteweg tensor defined by Dunn \& Serrin \cite{DS}. \end{proof}
\section{Without phase transformation} 
\label{sec:nsch}
The Navier-Stokes-Cahn-Hilliard equations (NSCH) read 
\begin{equation} \label{eq:nsch}
\begin{split}
\del{}{t} \rho + \divv( \rho \bu) & = 0, 
\\
\del{}{t} (\rho \bu) + \divv ( \rho \bu \otimes \bu ) - \divv \mathcal{S}(\bu)
- \divv \mathcal{P}(c) & = 0 
\\
\del{}{t} ( c \rho ) + \divv ( c \rho \bu ) - \divv ( \gamma \nabla \mu ) & = 0.
\end{split}
\end{equation}
We assume again (2) - (6) and use 
the chemical potential 
\begin{equation} \label{law:1}
\begin{split}
\mu & = \frac{1}{\theta} \del{}{c} G(p,c) - \rho^{-1} \nabla \cdot ( \delta \rho \nabla c ),
\end{split}
\end{equation}
as well as a mobility $\gamma>0$. As for NSAC, the pressure can be eliminated from the equations,
turning them into a Navier-Stokes-Korteweg system.  
This Navier-Stokes-Korteweg system has the same capillarity $\kappa(\rho)$ and extended Helmholtz 
energy $\psi$ as the ones found in the NSAC case (equation \eqref{eq:psi});
differently from the NSAC case \eqref{eq:nsac}, the Cauchy tensor acquires a 
non-local part.
\begin{thm} \label{thm:nsch}
NSCH can be written as the Navier-Stokes-Korteweg system
\begin{equation} \label{eq:NSK-2}
\begin{split}
\partial_t \rho + \divv ( \rho \bu ) & = 0, 
\\
\partial_t (\rho \bu)  + \divv ( \rho \bu \otimes\bu )
- 
\divv (\mathcal{S}_\gamma(\bu) + \mathcal{K}(\rho)) & = 0 
\end{split}
\end{equation}
with
\begin{equation} \label{eq:S-gamma}
\mathcal{S}_\gamma(\bu) = 2 \mu \mathcal{D}(\bu) + \lambda \divv \bu \, \mathcal{I} + \frac{\theta}{(\Delta\tau)^2} \Lambda^{-1}_\gamma( \divv \bu) \mathcal{I}, 
\end{equation}
$\mathcal{K}(\rho)$ given as in \eqref{eq:K}, and $\Lambda_\gamma^{-1}$ 
a solution operator
for the elliptic problem
\begin{equation} \label{eq:Lambda-3}
\Lambda_\gamma \phi \equiv- \divv ( \gamma \nabla \phi) = f  
\end{equation}
\end{thm}
\begin{remark}
\rm
In case of the whole space $\R^n$, one may take 
\begin{align*}
\Lambda^{-1}_\gamma f := \frac{1}{\gamma} (\Phi_n * f)(x) = \int_{\R^n} \Phi_n(x-y) f(y) \, dy
\end{align*}
with $\Phi_n$ the Poisson kernel.
On bounded domains $\Omega$ one may fix 
$\Lambda^{-1}_\gamma$ using the constraints 
\begin{align*} 
\del{}{\nu} \phi = 0, \: \int_\Omega \phi \, dx = 0. 
\end{align*}
together with the boundary condition $$\ska{\bu}{\nu}{|\partial\Omega}=0$$ that yields 
$\int_\Omega f \, dx = 0$ for $f=  \divv \bu $.
\end{remark}
\begin{proof}
The only difference from the proof of Theorem 1 is in 
the representation of the pressure.
Using \eqref{eq:hat-c} in the Cahn-Hilliard equation of \eqref{eq:nsch}, we obtain this time
\begin{align*}
\frac{1}{\Delta \tau} \divv \bu = - \divv ( \gamma \nabla (- \mu ) ), \quad \mu = \frac{\Delta \tau}{\theta} p + W'(c) - \rho^{-1} ( \delta \rho \nabla c )  
\end{align*}
and thus
\begin{align*}
p&=-\frac{\theta}{(\Delta\tau)^2} \Lambda_\gamma^{-1} \divv \bu  - \frac{\theta}{\Delta \tau} W'(c) + \rho^{-1} \divv 
\big ( \frac{\delta \theta}{\Delta \tau} \rho \nabla c \big) 
\\
& =  -\frac{\theta}{(\Delta\tau)^2} \Lambda_\gamma^{-1} \divv \bu + 
\rho^2 R'(\rho) -  \frac{1}{\rho} \divv \left( \frac{\delta^*}{\rho} \nabla \rho \right).
\end{align*}
\end{proof}


\begin{thebibliography}{99}
\bibitem{BG} 
S. Benzoni-Gavage, R. Danchin, S. Descombes, D. Jamet: 
Structure of Korteweg models and stability of diffuse interfaces. 
\textit{ Interfaces Free Bound.}\ {\bf 7} (2005), 371--414.

\bibitem{B}
T. Blesgen: A generalisation of the Navier-Stokes equations to two-phase-flows. 
\textit{ J. Phys.\ D: Appl.\ Phys.}\ {\bf 32} (1999), 1119--1123.

\bibitem{DS}  J.E. Dunn, J. Serrin: \textit{ On the thermomechanics of interstitial 
working.} Arch. rational Mech. Anal., 88(2), 95-133 (1985).

\bibitem{F3}
H. Freist\"uhler: 
Phase transitions and traveling waves in compressible fluids.
\textit{ Arch.\ Rational Mech.\ Anal.} (2013).

\bibitem{FK}
H. Freist\"uhler and M. Kotschote:
Diffuse planar phase boundaries in a two-phase fluid with one incompressible phase.
{\tt arxiv 1306.1905}

\bibitem{MK3}
M. Kotschote: Strong solutions to the Navier-Stokes equations
for a compressible fluid of Allen-Cahn type.
\textit{ Arch.\ Rational Mech.\ Anal.} {\bf 206} (2012), 489--514.

\bibitem{LT} J. Lowengrub and L. Truskinovsky: 
Quasi-incompressible Cahn-Hilliard fluids and topological transitions. 
\textit{ R.\ Soc.\ Lond.\ Proc.\ Ser.\ A Math.\ Phys.\ Eng.\ Sci.}\ 454 (1998) 2617--2654. 

\bibitem{S} M. Slemrod: 
Admissibility criteria for propagating phase boundaries in a van der Waals fluid. 
\textit{ Arch.\ Rational Mech.\ Anal.}\ 81 (1983), 301--315.
\end{thebibliography}
\end{document}